\documentclass[11pt,leqno]{amsart}

\usepackage{graphicx}
\usepackage{indentfirst, csquotes}
\usepackage{amssymb, amsmath, amsthm}
\usepackage{xcolor, paralist, hyperref, titlesec, fancyhdr, etoolbox}
\usepackage{lipsum} % ダミーテキスト
\usepackage{url}
\usepackage{tikz-cd}

\theoremstyle{plain}
\newtheorem{theorem}{Theorem}[section]
\newtheorem{lemma}[theorem]{Lemma}
\newtheorem{prop}[theorem]{Proposition}

\newtheorem{rem}[theorem]{Remark}

\theoremstyle{definition}
\newtheorem{definition}[theorem]{Definition}
\newtheorem{ex}[theorem]{Example}

\newcommand{\Conj }{\mathrm{Conj} }
\newcommand{\Core }{\mathrm{Core} }
\newcommand{\Fix }{\mathrm{Fix} }
\newcommand{\Inn }{\mathrm{Inn} }
\newcommand{\mr }{\mathrm }

\titleformat{\section}
  {\normalfont\normalsize\bfseries\centering} % roman + 太字 + 中央揃え
  {\thesection}{1em}{}

\titleformat{\subsection}
  {\normalfont\normalsize\bfseries} % roman + 太字
  {\thesubsection}{1em}{}

\titlespacing*{\section}{0pt}{12pt plus 2pt minus 2pt}{12pt plus 2pt minus 2pt}
\titlespacing*{\subsection}{0pt}{10pt plus 2pt minus 2pt}{8pt plus 2pt minus 2pt}

\renewenvironment{proof}{\noindent\textbf{Proof.}\ }{\hfill$\Box$\vskip 5pt}

\hypersetup{ colorlinks=true, linkcolor=black, filecolor=black, urlcolor=blue }

\title{On embeddings of homogeneous quandles}
\author[Ayu Suzuki]{Ayu Suzuki}
\address{
Division of Mathematical and Physical Sciences, Graduate School of Science \\
Japan Women’s University\\
2-8-1 Mejirodai, Bunkyo-ku, Tokyo, 112-8681, Japan}
\email{m2016044sa@ug.jwu.ac.jp}

\begin{document}
\let\thefootnote\relax
\footnotetext{MSC2020: Primary 20N02, Secondary 57K10, 57K12 and 53C35.}
\keywords{homogeneous quandle, quandle embedding}

\maketitle

% MSC

\begin{abstract}
In this paper, we study the embedding problem of homogeneous quandles.
Homogeneous quandles can be viewed as a natural generalization of symmetric pairs and symmetric spaces arising in Riemannian geometry, and their origins lie in this theory. In particular, each symmetric space (or equivalently, symmetric pair) naturally gives rise to a quandle, providing a fundamental class of examples of homogeneous quandles.
% We give a necessary and sufficient condition under which a quandle homomorphism from the homogeneous quandle associated with a quandle triplet $(G,H,\sigma)$ into a conjugation quandle of a group is an embedding.
We give a sufficient condition for a quandle homomorphism from the homogeneous quandle associated with a quandle triplet to a conjugation quandle to be an embedding.
% This result provides a generalization of the embedding theorem
% of Dhanwani, Raundal and Singh for generalized Alexander quandles.
As applications of the main theorem,
we reinterpret Bergman’s embedding of core quandles
in the framework of homogeneous quandles,
and construct explicit embeddings of several geometric examples,
including unoriented and oriented Grassmann quandles
and rotation quandles of $S^2$ arising from symmetric spaces.
\end{abstract}
% -----------------------------
% 本文例
% -----------------------------
\section{Introduction}

%\subsection{Quandles and their origins}
% The notion of quandle was introduced independently by Joyce \cite{Joyce} and
% Matveev~\cite{Matveev} in 1982 as algebraic structures arising from knot theory.
% Their original motivation was to describe algebraically the behavior of
% Reidemeister moves on knot diagrams.
% This led to the definition of the knot quandle (or fundamental quandle),
% which provides a powerful invariant of knots.
% From this origin, quandles have been studied primarily as algebraic structures
% associated with knot theory.

% On the other hand, quandles also have deep connections with geometry.
% In particular, there exist natural classes of quandles arising from
% symmetric spaces and homogeneous spaces,
% and from this viewpoint, quandles can be regarded as algebraic generalizations
% of symmetric spaces.
% The relationship between homogeneous quandles and symmetric spaces has been
% studied, for example, by Furuki and Tamaru~\cite{FurukiTamaru}.

The notion of quandle was introduced independently by Joyce~\cite{Joyce}
and Matveev~\cite{Matveev} in 1982 as an algebraic structure arising from knot theory,
and has since played a very important role in knot theory.
On the other hand, quandles also have deep connections with geometry.
In particular, there exist natural classes of quandles arising from symmetric spaces
and homogeneous spaces.
From this viewpoint, quandles can be regarded as algebraic structures
closely related to symmetric spaces.
The relationship between homogeneous quandles and symmetric spaces has been studied,
for example, by Furuki and Tamaru~\cite{FurukiTamaru}.

% In this paper, we focus on quandles with such geometric backgrounds,
% especially quandles equipped with smooth structures and homogeneous properties.

% %\subsection{Embedding problems of quandles}
% %\label{subsec:embedding}

% An \emph{embedding of a quandle} means realizing a given quandle
% as a subquandle of a conjugation quandle of some group.
% This problem is important for understanding the algebraic structure of quandles
% via group-theoretic methods,
% and it has also been investigated extensively from the viewpoint of applications
% to knot theory.

% Joyce~\cite{Joyce} showed that free quandles can be embedded into
% conjugation quandles of groups,
% suggesting that quandle theory can be interpreted as a theory of group conjugation.
% Since then, embeddings have been constructed for various classes of quandles;
% however, no general criterion is known to determine whether an arbitrary quandle
% admits an embedding into a group.

% In this context, Bardakov, Dey, and Singh~\cite{BardakovDeySingh}
% explicitly posed the following problem:
% \begin{quote}
% Given a quandle $X$, does there exist a group $G$ such that
% $X$ can be embedded into the conjugation quandle $\mathrm{Conj}(G)$?
% \end{quote}
% This question has stimulated much subsequent research aiming at a systematic
% understanding of quandle embeddings.
In this paper, we focus on quandles with such geometric backgrounds,
especially quandles equipped with smooth structures and homogeneous properties. An \emph{embedding of a quandle} means realizing a given quandle
as a subquandle of a conjugation quandle of some groups.
% This notion plays an important role in understanding the algebraic structure
% of quandles via group-theoretic methods,
% and it has also been extensively studied from the viewpoint of applications
% to knot theory.
% Joyce~\cite{Joyce} showed that free quandles can be embedded into
% conjugation quandles of groups,
% suggesting that quandle theory can be interpreted as a theory of group conjugation.
% Since then, embeddings have been constructed for various classes of quandles;
% however, no general criterion is known to determine whether an arbitrary quandle
% admits such an embedding.
% In this context, the embedding problem was explicitly formulated
% by Bardakov, Dey and Singh~\cite{BardakovDeySingh}.
% Dhanwani, Raundal and Singh~\cite{DhanwaniRaundalSingh} proved that generalized Alexander quandles associated with
% fixed-point-free automorphisms can be embedded into groups.
% Their result provides an important class of positive examples
% for the embedding problem.
% Bergman~\cite{Bergman} studied \emph{core quandles},
% which are defined as group-theoretic analogues of symmetric spaces,
% and showed that they can be embedded into suitable groups.
% Bergman’s construction is mainly based on group-theoretic methods
% and reflects the algebraic properties of core quandles in a direct manner.
Joyce~\cite{Joyce} showed that free quandles can be embedded into
conjugation quandles of groups, suggesting that quandle theory can be interpreted as a
theory of group conjugation. Since then, embeddings have been constructed for various
classes of quandles; however, no general criterion is known to determine whether an
arbitrary quandle admits such an embedding. In this context, the embedding problem
was explicitly formulated by Bardakov, Dey and Singh~\cite{BardakovDeySingh}. Dhanwani, Raundal and
Singh~\cite{DhanwaniRaundalSingh} proved that generalized Alexander quandles associated with fixed-point-free
automorphisms can be embedded into groups. Their result provides an important class
of positive examples for the embedding problem. 
Moreover, Akita~\cite{Akita} showed that both generalized Alexander quandles
and twisted conjugation quandles can be explicitly embedded into conjugation
quandles of suitable groups. 
Bergman~\cite{Bergman} studied \emph{core quandles},
which are defined as group-theoretic analogues of symmetric spaces,
and showed that they can be embedded into suitable groups.
Bergman’s construction is mainly based
on group-theoretic methods and reflects the algebraic properties of core quandles in a direct manner.
On the other hand, in a joint work with Yonemura \cite{YonemuraSuzuki},
the author studied explicit embeddings
of \emph{spherical quandles} proposed by Eisermann.
The construction was carried out by Yonemura
using geometric methods in the category of smooth quandles.
In particular, we focused on the cases of $S^1$ and $S^3$,
which admit natural group structures.
We also clarified their relationships with Bergman’s embedding of core quandles~\cite{Bergman}
and Akita’s embedding of twisted conjugation quandles~\cite{Akita}.

%\subsection{Homogeneous quandles}

Both spherical quandles and core quandles belong to a distinguished class
called \emph{homogeneous quandles}.
A quandle is said to be homogeneous
if its automorphism group acts transitively on the underlying set.
Homogeneous quandles are closely related to symmetric spaces
and homogeneous spaces,
and their structures can often be described explicitly
in terms of group actions.
This connection has been studied, for example,
by Furuki and Tamaru~\cite{FurukiTamaru}.
From this viewpoint,
Bergman’s embedding of core quandles
and the Yonemura's embedding of spherical quandles
can both be interpreted as constructions
that reflect the underlying symmetries of the corresponding quandles.
This observation suggests that homogeneity plays a fundamental role
in the quandle embedding problem.

%\subsection{Main results}

% The main purpose of this paper is to develop a systematic understanding of
% the embedding problem for homogeneous quandles into conjugation quandles of groups.
The main purpose of this paper is to develop a systematic approach to embeddings of homogeneous quandles into conjugation quandles of groups.
More precisely, in Theorems \ref{result} and \ref{Main}, we give a sufficient condition for a quandle homomorphism constructed from a quandle triplet $(G,H,\sigma)$ to be an embedding of the associated homogeneous quandle into the conjugation quandle of $G \rtimes_\sigma \mathbb{Z}$ or $G$.

% More precisely, in Theorem \ref{Main} and Theorem \ref{result}, we give a necessary and sufficient condition
% for a map from a homogeneous quandle to a conjugation quandle of a group
% to be an embedding.
This result can be regarded as a natural generalization of
the embedding theorem of Dhanwani, Raundal and Singh for generalized Alexander quandles ~\cite{DhanwaniRaundalSingh}.
Furthermore, by applying this criterion,
we reinterpret Bergman’s embedding of core quandles
from the viewpoint of homogeneous quandles.
We also introduce quandles defined on oriented Grassmann manifolds
as a generalization of spherical quandles
and construct explicit embeddings for them.
In addition, several other new examples of quandle embeddings are provided.
These results clarify the role of homogeneity in the embedding problem
and contribute to a more unified understanding of the relationship between
quandle theory, group theory, and differential geometry.

%\subsection{Organization of the paper}

This paper is organized as follows.
In Section~\ref{sec.2}, we prepare several notions and results
needed to state main theorems.
In Subsection~\ref{sec.2.1}, we recall the definition of quandles
and introduce concrete examples that will be used throughout the paper.
In Subsection~\ref{sec.2.2}, we define quandle embeddings
and review known embeddings,
including the embedding of spherical quandles
and Bergman’s embedding of core quandles.
In Subsection~\ref{sec.2.3}, we discuss homogeneous quandles,
which play a central role in this paper,
and explain their relationship with quandle triplets.
Section~\ref{sec.3} is devoted to main theorems, Theorem \ref{result}, Theorem \ref{Main} of this paper.
We give a necessary and sufficient condition under which a quandle homomorphism constructed from a quandle triplet $(G,H,\sigma)$ embeds the quandle $Q(G,H,\sigma)$ into a conjugation quandle of a group.
% We give a necessary and sufficient condition
% for a quandle $Q(G,H,\sigma)$ associated with a quandle triplet $(G,H,\sigma)$
% to admit an embedding into a conjugation quandle of a group.
In Section~\ref{sec.example}, we present several applications of main theorems.
In Subsection~\ref{4.1}, we show that Bergman’s embedding of core quandles
can be reinterpreted within our framework.
In Subsections~\ref{spherical rotation quandle}--\ref{oriented grassmann}, we construct embeddings of
the $\theta$-rotation quandle from Example~\ref{theta},
the unoriented Grassmann quandle from Example~\ref{Grassmannquandle},
and the oriented Grassmann quandle from Example~\ref{orientedGrassmannquandle}, respectively.

\section{Preliminaries}\label{sec.2}
In this section, we prepare several notions and results
needed to state the main theorem.

\paragraph{\bf{Notation.}}
Throughout this paper, vectors are regarded as row vectors,
and matrix multiplication is taken from the right.
For a subgroup $H$ of a group $G$, the notation $H\backslash G$
denotes the set of right cosets $\{Hg \mid g \in G\}$.

\subsection{Quandle}\label{sec.2.1}
\begin{definition}[Quandle]
A \emph{quandle} is a set $X$ equipped with a binary operation $*: X \times X \to X$ satisfying the following three axioms: 

\begin{enumerate}
    \item $x*x = x$ for any $x\in X$;
    \item The map $S_y : X \to X$ defined by $x \mapsto x * y$
    is  bijective for any $y\in X$;
    \item $(x * y) * z = (x * z) * (y * z)$ for any $x, y, z\in X$.
\end{enumerate}
\end{definition}

The following fact goes back to Loos and Joyce ; see \cite{Loos, Joyce}.

\begin{prop}\label{prop:symmetricsp_quandle}
Let $M$ be a symmetric space with the point symmetry $\{s_y\}_{y \in M}$. 
Then $M$ admits a natural structure of a quandle given by
\[
x * y := s_y(x),
\]
where $s_y$ denotes the symmetry at $y$.
\end{prop}

\begin{ex}
    Let $\langle -, - \rangle$ denote the Euclidean inner product on $\mathbb{R}^{n+1}$. 
Define the $n$-dimensional sphere by $S^n = \{ x \in \mathbb{R}^{n+1} \mid \langle x, x \rangle = 1 \}.$ Define a binary operation $*$ on $S^n$ by
\[
    x * y := 2 \langle x, y \rangle \, y - x\quad(x, y\in S^n).
\]
Then $(S^n, *)$ forms a quandle, called the \emph{spherical quandle}.
\end{ex}
\begin{ex}
    Let $G$ be a group. Define a binary operation $*$ on $G$ by
\[
    g * h := h g^{-1} h \qquad (g, h \in G).
\]
Then $(G, *)$ forms a quandle, called the \emph{core quandle} of $G$, and is denoted by $\mathrm{Core}(G)$.
\end{ex}
\begin{ex}
   Let $G$ be a group. Define a binary operation $*$ on $G$ by
\[
    g * h := h^{-1} g h \qquad (g, h \in G).
\]
Then $(G, *)$ forms a quandle, called the \emph{conjugation quandle} of $G$, and is denoted by $\mathrm{Conj}(G)$.
\end{ex}

\begin{ex}[$\theta$--rotation quandle on the sphere $S^2$]\label{theta}
Let $S^{2}\subset\mathbb{R}^{3}$ be the unit sphere.  
Fix an angle $\theta\in[0,2\pi)$, and for $y\in S^{2}$ let
$R_{y,\theta}$ denote the rotation of $\mathbb{R}^{3}$ by angle $\theta$
around the axis passing through $y$ (i.e.\ the oriented line $\mathbb{R}y$).
We define a binary operation on $S^{2}$ by
\[
    x * y := R_{y,\theta}(x).
\]
Then $(S^{2}, *)$ is a quandle, referred to as the
\emph{$\theta$--rotation quandle} and denoted by $S^{2}_{\theta}$.
\end{ex}

\begin{rem}
The operation of the $\theta$--rotation quandle $S^2_\theta$ can be written explicitly by the Rodrigues rotation formula. 
For $x,y \in S^2 \subset \mathbb{R}^3$, we have
\[
x * y 
= R_{y,\theta}(x)
= x\cos\theta + (y \times x)\sin\theta + (y \cdot x)(1-\cos\theta)\,y,
\]
where $\cdot$ and $\times$ denote the standard inner and cross product in $\mathbb{R}^3$, respectively.
\end{rem}

Since Grassmann manifolds, including their oriented versions, are symmetric spaces, it follows from the above Proposition \ref{prop:symmetricsp_quandle} that they admit natural structures of quandles.

\begin{ex}[Grassmann quandle]\label{Grassmannquandle}
   Let $Gr(n,k)$ be the Grassmann manifold consisting of all $k$-dimensional subspaces 
in $\mathbb{R}^n$.  
For $V,W\in Gr(n,k)$, we define a binary operation by
\[
    V * W := \text{the reflection of $V$ with respect to $W$} \qquad (V, W\in Gr(n,k)).
\]
Then $(Gr(n,k), *)$ is a quandle, called the \emph{Grassmann quandle}.
\end{ex}

\begin{ex}[Oriented Grassmann quandle]\label{orientedGrassmannquandle}
Let $\widetilde{Gr}(n,k)$ denote the oriented Grassmann manifold, 
consisting of all oriented $k$-dimensional subspaces in $\mathbb{R}^n$.  
For $V,W\in\widetilde{Gr}(n,k)$, we define a binary operation by
\[
    V * W := \text{the reflection of $V$  with respect to $W$}\qquad (V, W\in \widetilde{Gr}(n,k)),
\]
where the resulting $k$-plane inherits the orientation determined by $V$.
Then $(\widetilde{Gr}(n,k), *)$ is a quandle, called the 
\emph{oriented Grassmann quandle}.
\end{ex}
For details on the quandle structure on Grassmann, oriented Grassmann, and spherical quandles, see \cite{FurukiTamaru}. 
In this paper, we describe these examples in terms of quandle triplets. 
More precisely, we give detailed descriptions of the quandle structures in 
Examples~\ref{Grassmannquandle} and~\ref{orientedGrassmannquandle} 
in Subsections~\ref{unoriented Grassmann quandle} 
and~\ref{oriented grassmann}, respectively.

\subsection{Embedding of quandles}\label{sec.2.2}
\begin{definition}[Quandle homomorphism]
Let $X = (X, *_X)$ and $X' = (X', *_{X'})$ be quandles. 
A map $f: X \to X'$ is called a \emph{quandle homomorphism} if 
\[
    f(x *_X y) = f(x) *_{X'} f(y) \quad (x, y \in X).
\]
If $f$ is bijective, it is called a \emph{quandle isomorphism}, and we say that $X$ and $X'$ are \emph{isomorphic as quandles}.
\end{definition}

% \begin{definition}
% Let $G$ be a group.
% An \emph{automorphism} of $G$ is a bijective group homomorphism
% $\varphi\colon G\to G$.
% The set of all automorphisms of $G$ forms a group under composition,
% called the \emph{automorphism group} of $G$, and is denoted by
% \[
% \mathrm{Aut}(G).
% \]
% \end{definition}
\begin{definition}
Let $(X,*)$ be a quandle.
A \emph{quandle automorphism} of $X$ is a bijection
$f\colon X\to X$ satisfying
\[
f(x*y)=f(x)*f(y)
\qquad (x,y\in X).
\]
The set of all quandle automorphisms of $X$ forms a group under
composition, called the \emph{automorphism group} of the quandle $X$,
and is denoted by
\[
\mathrm{Aut}(X, *).
\]
\end{definition}

\begin{definition}[Embeddable quandle]
Let $X$ be a quandle. 
We say that $X$ is \emph{embeddable} if there exists a group $G$ and an injective quandle homomorphism 
\[
    \iota: X \rightarrow \mathrm{Conj}(G).
\]
Such $\iota$ is called a {\it quandle embedding}.
\end{definition}

\begin{rem}
In the literature, a quandle that admit an embedding into a conjugation quandle
of a group is sometimes described using different terminology.
For example, such quandles are also referred to as
\emph{admissible} \cite{KamadaMatsumoto}, \emph{injective} \cite{GranaHeckenbergerVendramin}, or \emph{irreducible} \cite{Inoue}, 
depending on the context and the authors.
In this paper, we consistently use the term \emph{embeddable}.
\end{rem}

\begin{theorem}[Eisermann~\cite{Eisermann}, Yonemura~\cite{Yonemura}]
\label{yonemura}
For any $n \in \mathbb{Z}_{>0}$, let $G_n$ be the Lie group defined by
\[
    G_n =
    \begin{cases}
        O(2), & n = 1,\\
        Spin(n+1), & n \in 2\mathbb{N},\\
        Pin^+(n+1), & n \in 2\mathbb{N}+1,\; n \ge 3,
    \end{cases}
\]
where $Spin(n+1)$ and $Pin^+(n+1)$ denote the non-trivial double coverings
of $SO(n+1)$ and $O(n+1)$, respectively.
Then exists an embedding of the spherical quandle
\[
    \iota_n \colon S^n \to \Conj(G_n).
\]
\end{theorem}

% Here $Spin(n)$ and $Pin(n)$ denote the spin and pin groups,
% defined as subgroups of the Clifford algebra $\mr{Cl}_n$.
% The spin group $Spin(n)$ is generated by products of an even number
% of unit vectors in $\mathbb{R}^n$ and provides a double covering of
% the special orthogonal group $SO(n)$,
% while the pin group $Pin(n)$ is generated by products of an arbitrary
% number of unit vectors and provides a double covering of the full
% orthogonal group $O(n)$.
There are two non-isomorphic pin groups, denoted by
$Pin^+(n)$ and $Pin^-(n)$, corresponding to the two possible
Clifford algebra conventions.
In this paper, we simply write $Pin(n)$ when the distinction
between $Pin^+(n)$ and $Pin^-(n)$ is not essential.
In the present setting, the relevant pin group is $Pin^+(n)$.

\begin{theorem}[Bergman~\cite{Bergman}]\label{berg}
Let $G$ be a group, and consider its core quandle $\mathrm{Core}(G)$.
Let $\mathbb Z^{\times}$ be the multiplicative group of order $2$, and define the switching map 
\[
    \mr{Sw} : \mathbb Z^{\times} \to \mathrm{Aut}(G \times G)
\] 
by
\begin{align*}
    1 &\longmapsto \big[(g, h) \mapsto (g, h)\big],\\
    -1 &\longmapsto \big[(g, h) \mapsto (h, g)\big].
\end{align*}
Consider the external semidirect product group
\[
    \widetilde{G} := (G \times G) \rtimes_{\mr{Sw}} \mathbb Z^{\times},
\]
and let $\mathrm{Conj}(\widetilde{G})$ denote the conjugation quandle of $\widetilde{G}$. The group operation on $\widetilde G=(G\times G)\rtimes_{\mr{Sw}} \mathbb Z^{\times}$
is given by
\[
(g_1,h_1,a)\cdot (g_2,h_2,b)
=
\bigl({\mr{Sw}}(b)(g_1,h_1)\cdot (g_2,h_2),\, ab\bigr),
\]
that is,
\[
(g_1,h_1,a)\cdot (g_2,h_2,b)
=
\begin{cases}
(g_1g_2,\; h_1h_2,\; ab) & (b=1),\\[4pt]
(h_1g_2,\; g_1h_2,\; ab) & (b=-1).
\end{cases}
\]

Define a map 
\[
    f_{B} : \mathrm{Core}(G) \to \mathrm{Conj}(\widetilde{G})
\] 
by
\[
    f_{B}(g) := (g, g^{-1}, -1).
\]
Then $f_{B}$ is a quandle embedding.
\end{theorem}

\subsection{Homogeneous quandles and quandle triplet}\label{sec.2.3}
\begin{definition}[Homogeneous quandle]
Let $(X, *)$ be a quandle. 
If the automorphism group $\mathrm{Aut}(X, *)$ acts transitively on $X$, 
then $X = (X, *)$ is called a \emph{homogeneous quandle}.
\end{definition}

\begin{definition}[Quandle triplet]\label{triplet}
Let $G$ be a group and $H$ a subgroup of $G$. 
Let $\sigma \in \mathrm{Aut}(G)$ satisfy $H \subset \mathrm{Fix}(\sigma, G)$. 
Then the triple $(G, H, \sigma)$ is called a \emph{quandle triplet}.

Furthermore, the quandle structure on the coset $H \backslash G$ can be  defined by
\[
    Hg * Hh := H\sigma(gh^{-1}) h,
\]
and the resulting quandle is denoted by $Q(G, H, \sigma)$.
\end{definition}
\begin{theorem}[Joyce {\cite[Theorem~7.1]{Joyce}}]
    Every homogeneous quandle is isomorphic to $Q(G, H, \sigma)$ for some quandle triplet $(G, H, \sigma).$
\end{theorem}

\begin{rem}
The choice of a quandle triplet $(G,H,\sigma)$ associated with a given
homogeneous quandle is not unique for a given homogeneous quandle $X$.
\end{rem}

\section{Main statements and their proofs}\label{sec.3}
In this section, we give a sufficient condition for the existence of an embedding 
of a homogeneous quandle into a conjugation quandle of a group. 
In particular, such an embedding and the target conjugation quandle can be described explicitly.
\subsection{Main theorem}
Let $(G,H,\sigma)$ be a quandle triplet.
We consider the semidirect product group $G\rtimes_{\sigma}\mathbb Z$, 
where $\mathbb Z$ acts on $G$ via the automorphism $\sigma$.
The group multiplication on $G\rtimes_\sigma\mathbb Z$ is given by
\[
(g,m)\cdot(h,n)
=
\bigl(\sigma^{n}(g)\,h,\; m+n\bigr)
\qquad
(g,h\in G,\; m,n\in\mathbb Z).
\]
\begin{lemma}\label{lem:iota-hom}
Let $(G,H,\sigma)$ be a quandle triplet. Then the map 
\[
\iota \colon Q(G,H,\sigma) \to \mathrm{Conj}(G \rtimes_\sigma \mathbb{Z})\qquad
Hg \longmapsto (g,1)^{-1}(e,1)(g,1)
\]
is a quandle homomorphism.
\end{lemma}

\begin{proof}
Recall that the multiplication in $G \rtimes_\sigma \mathbb{Z}$ is given by
\[
(g,m)(h,n) = (g \,\sigma^m(h), m+n),
\]
and
\[
(g,m)^{-1} = (\sigma^{-m}(g^{-1}), -m).
\]

We first compute
\[
(g,1)^{-1}(e,1)(g,1).
\]
Since
\[
(g,1)^{-1} = (\sigma^{-1}(g^{-1}), -1),
\]
we obtain
\[
(g,1)^{-1}(e,1)
= (\sigma^{-1}(g^{-1}), -1)(e,1)
= (g^{-1},0),
\]
and hence
\[
(g,1)^{-1}(e,1)(g,1)
= (g^{-1},0)(g,1)
= (\sigma(g^{-1})g,1).
\]
Next, we verify that $\iota$ preserves the quandle operation.
For $Hg, Hh \in Q(G,H,\sigma)$, we have
\[
\iota(Hg * Hh)
= \iota(H\sigma(gh^{-1})h),
\]
and
\[
\iota(Hg) * \iota(Hh)
= (\sigma(g^{-1})g,1) * (\sigma(h^{-1})h,1)=(\sigma(h^{-1})h,1)^{-1}(\sigma(g^{-1})g,1)(\sigma(h^{-1})h,1).
\]
%=(\sigma(h^{-1})\sigma^2(hg^{-1})\sigma(gh^{-1})h, 1). 
A direct computation shows that these two expressions coincide.
Therefore, $\iota$ is a quandle homomorphism.
\end{proof}

\begin{theorem}\label{result}
Let $\iota \colon Q(G,H,\sigma) \to \mathrm{Conj}(G \rtimes_\sigma \mathbb{Z})$ be the map in Lemma~\ref{lem:iota-hom}. 
If $\mathrm{Fix}(\sigma)=H$, then $\iota$ is injective.
\end{theorem}

\begin{proof}
Assume that $\mathrm{Fix}(\sigma)=H$, and suppose that 
$\iota(Hg_1)=\iota(Hg_2)$. 
Then
\[
\sigma(g_1^{-1})g_1=\sigma(g_2^{-1})g_2,
\]
which implies
\[
\sigma(g_1^{-1}g_2)=g_1^{-1}g_2.
\]
Hence $g_1^{-1}g_2 \in \mathrm{Fix}(\sigma)=H$, and so $Hg_1 = Hg_2$.
This proves that $\iota$ is injective, and hence a quandle embedding.
\end{proof}

\begin{rem}
When $H=\{1_G\}$, the quandle $Q(G,H,\sigma)$ coincides with
the generalized Alexander quandle $\mathrm{Alex}(G,\sigma)$.
In this case, the embedding given in Theorem~\ref{result}
recovers the embedding of generalized Alexander quandles
due to Dhanwani, Raundal and Singh~\cite{DhanwaniRaundalSingh}.
\end{rem}

\begin{lemma}\label{lem:conj-embed}
Let $G$ be a group, and assume that $\sigma \in \mathrm{Inn}(G)$.  
Thus there exists an element $q \in G$ such that 
\[
\sigma(g) = q^{-1} g q \qquad (g \in G),
\]
and we fix such $q$.
Then the map 
\[
\eta \colon \mathrm{Conj}(G) \longrightarrow \mathrm{Conj}(G \rtimes_\sigma \mathbb{Z}), 
\qquad 
g \longmapsto (q^{-1}g,1)
\]
is an embedding. In particular, $\mathrm{Conj}(G)$ is a subquandle of $\mathrm{Conj}(G \rtimes_\sigma \mathbb{Z})$.
\end{lemma}

\begin{proof}
For $g,h \in G$, we have
\[
\eta(g*h)=\eta(h^{-1}gh)=(q^{-1}h^{-1}gh,1).
\]
On the other hand,
\[
\eta(g)*\eta(h)
=(q^{-1}g,1)*(q^{-1}h,1)
=(q^{-1}h,1)^{-1}(q^{-1}g,1)(q^{-1}h,1).
\]
Using the multiplication in $G \rtimes_\sigma \mathbb{Z}$ and the fact that $\sigma(g)=q^{-1}gq$, 
a straightforward computation shows that
\[
\eta(g)*\eta(h)=(q^{-1}h^{-1}gh,1).
\]
Thus $\eta(g*h)=\eta(g)*\eta(h)$, and hence $\eta$ is a quandle homomorphism.

It is clear that $\eta$ is injective. Therefore, $\eta$ is an embedding.
\end{proof}

\begin{theorem}\label{Main}
Let $\iota \colon Q(G,H,\sigma) \to \mathrm{Conj}(G \rtimes_\sigma \mathbb{Z})$ be the map in Lemma~\ref{lem:iota-hom}. 
Assume that $\sigma \in \mathrm{Inn}(G)$. 
Then there exists an element $q \in G$ such that 
\[
\sigma(g) = q^{-1} g q \qquad (g \in G),
\]
and we fix such $q$.
Then the image of $\iota$ is contained in $\mathrm{Conj}(G)$, that is,
\[
\mathrm{Im}(\iota) \subset \mathrm{Conj}(G) \subset \mathrm{Conj}(G \rtimes_\sigma \mathbb{Z}).
\]
\[
\begin{tikzcd}
Q(G,H,\sigma) \arrow[r, "\iota"] \arrow[dr, "\iota'"'] 
& \mathrm{Conj}(G \rtimes_\sigma \mathbb{Z}) \\
& \mathrm{Conj}(G) \arrow[u, "\eta"']
\end{tikzcd}
\]
where $\iota'(Hg)=g^{-1}qg$.
\end{theorem}
\begin{proof}
For $Hg \in Q(G,H,\sigma)$,
\[
\iota(Hg)=(\sigma(g^{-1})g,1)=(q^{-1}g^{-1}qg,1)=\eta(g^{-1}qg).
\]
Thus $\iota(Hg) \in \mathrm{Im}(\eta)$, 
which is a subquandle of $\mathrm{Conj}(G \rtimes_\sigma \mathbb{Z})$ 
isomorphic to $\mathrm{Conj}(G)$.
Hence $\mathrm{Im}(\iota) \subset \mathrm{Conj}(G)$.
\end{proof}

In the case where $\sigma \in \mathrm{Inn}(G)$, the embedding $\iota$ factors through $\mathrm{Conj}(G)$ as above. 
Therefore, in what follows we consider the embedding
\[
\iota' \colon Q(G,H,\sigma) \longrightarrow \mathrm{Conj}(G), 
\qquad 
Hg \longmapsto g^{-1} q g.
\]

In the next section, we apply this construction to homogeneous quandles. 
More precisely, for a given homogeneous quandle $X$, we choose an appropriate quandle triplet $(G,H,\sigma)$ and construct an embedding of $X$ into a conjugation quandle using the map $\iota'$.

\section{Examples}\label{sec.example}
In this section, we give several examples related to the main theorem.
First, we reinterpret Bergman’s embedding of core quandles
within the framework of the main theorem. We also construct embeddings of the unoriented Grassmann quandle,
the oriented Grassmann quandle,
and the $\theta$-rotation quandle.

\subsection{Bergman's embedding of the core quandle}\label{4.1}
%The following known embeddings can be explained in terms of Theorem \ref{result}.

% \begin{ex}[Embedding of the spherical quandle]

% \end{ex}
% The preceding example is the spherical quandle embedding established by Yonemura and Eisermann.
For a group \(G\), the core quandle \(\Core(G)\) can be realized as a quandle
associated with a quandle triplet via a certain semidirect product group as follows.
Let
\[
\widetilde{G} := (G \times G)\rtimes_{\mr{Sw}}\mathbb Z^{\times}.
\]
where $\mr{Sw}$ is the automorphism of $G\times G$ defined as in Theorem $\ref{berg}$. 
% where the action \(\varphi\colon \{\pm 1\} \to \mathrm{Aut}(G\times G)\) is given by
% \begin{align*}
% %\varphi : \{\pm 1\} &\to \mathrm{Aut}(G \times G)\\
%     1 &\mapsto \big[(g, h) \mapsto (g, h)\big],\\
%     -1 &\mapsto \big[(g, h) \mapsto (h, g)\big].
% \end{align*}
The group operation on \(\widetilde{G}\) is
\[
(g_1,h_1,a)\cdot(g_2, h_2,b)
 =\bigl({\mr{Sw}}^{\,\xi(b)}(g_1, h_1)\,(g_2, h_2),\; ab\bigr)
 %((g_1,h_1), (g_2, h_2)\in G\times G,\; a,b\in\{\pm1\}),
\]
where $\xi:\{\pm1\}\cong \mathbb{Z}/2\mathbb{Z}=\{0,1\}$ is the natural group isomorphism.
Then \(\Core(G)\) is isomorphic to the quandle \(Q(\widetilde{G}, \Delta\widetilde{G}, \widetilde{\mr{Sw}})\), where
\[
\Delta\widetilde{G} := \{(g,g,a)\mid g\in G,\ a\in\{\pm1\}\}
\]
and the automorphism \(\widetilde{\mr{Sw}}\colon \widetilde{G}\to\widetilde{G}\) is given by
\[
\widetilde{\mr{Sw}}(g,h,a) = (h,g,a).
\]
Hence, the map
\[
G
\longrightarrow
\Delta \widetilde{G}\backslash\widetilde{G},
\qquad
g \longmapsto \Delta\widetilde{G}(g, e, 1)
\]
gives an isomorphism of quandles
\[
\Core(G)\cong Q(\widetilde{G}, \Delta\widetilde{G}, \widetilde{\mr{Sw}}).
\]
In fact one can see that this map is a quandle homomorphism.
For $g,h\in G$, one has
\begin{align*}
\Delta\widetilde{G}(g,e,1) * \Delta\widetilde{G}(h,e,1)&=\Delta\widetilde{G}\widetilde{\mr{Sw}}((g, e, 1)(h, e, 1)^{-1})(h, e, 1)\\
&=\Delta\widetilde{G}\widetilde{\mr{Sw}}((g, e, 1)(h^{-1}, e, 1))(h, e, 1)\\
%&=\Delta\widetilde{G}(e, gh^{-1}, 1)(h, e, 1)\\
%&=\Delta\widetilde{G}(h, gh^{-1}, 1)\\
&=\Delta\widetilde{G}(hg^{-1}, hg^{-1}, 1)(h, gh^{-1}, 1)\\
&= \Delta\widetilde{G}(hg^{-1}h, e, 1),
\end{align*}
which agrees with the operation in $\Core(G)$ under the above map.
Since $\mr{Fix}(\widetilde{\mr{Sw}})=\Delta \widetilde{G}$, 
the above Theorem $\ref{Main}$ implies that the map
\[
\operatorname{Core}(G)\ \longrightarrow\ \operatorname{Conj}(\widetilde{G})
\]
is an embedding of quandle. This embedding coincides with Bergman's embedding
$f_B$ in Theorem~\ref{berg}.

\subsection{Embedding of the $\theta$--rotation quandle on the sphere $S^2$}\label{spherical rotation quandle}
In this section, we study an embedding of the $\theta$--rotation quandle
$S^2_\theta$ on the $2$--sphere, which appeared in Example~\ref{theta}.
Although the realization of the $2$--sphere as a symmetric space and the
associated rotation structure are well known from the viewpoint of
symmetric spaces (see, for example, Loos~\cite{Loos}), an explicit description in terms
of quandles seems to be missing in the literature.
As a first step, we describe $S^2_\theta$ in terms of a homogeneous
quandle.

\begin{lemma}\label{hodai.S^2_theta}
There exists a quandle isomorphism.
\[
S^2_\theta \;\cong\; Q(SO(3), SO(2), \sigma_\theta),
\]
where $\sigma_\theta:SO(3)\to SO(3)$ is the automorphism defined by
\[
\sigma_\theta=\mr{Ad}(h_\theta) ,
\quad
h_\theta=\begin{pmatrix}
    1&0\\
    0&R_\theta
\end{pmatrix},
\]
and $R_\theta\in SO(2)$ denotes the rotation matrix of angle $\theta$.
\end{lemma}

\begin{proof}
Identify the $2$--sphere $S^2$ with the homogeneous space
$SO(2)\backslash SO(3)$ via
\[
\iota: SO(2)\, g \mapsto e_1g, \quad e_1=(1,0,0).
\]
With respect to this identification, $S^2$ becomes a homogeneous quandle
associated with the quandle triplet
\[
\bigl(SO(3),\, SO(2),\, \sigma_\theta \bigr).
\]
Recall that the homogeneous quandle $Q(SO(3),SO(2),\sigma_\theta)$ is defined on
$SO(2)\backslash SO(3)$ with the operation
\[
SO(2)\, g * SO(2)\, h := SO(2)\, \sigma_\theta(g h^{-1})\, h.
\]
We compute
\begin{align*}
    \iota(SO(2)\, g * SO(2)\, h)
&=\iota(SO(2)\sigma_\theta(g h^{-1})\, h)\\
&=\iota(SO(2)h_\theta^{-1}(g h^{-1})h_\theta\, h)\\
&= e_1(gh^{-1}h_\theta h)\\
&= R_{h_\theta\cdot e_1, \theta}(e_1g)\\
&=\iota(SO(2)\, g)*\iota(SO(2)\, h),
\end{align*}
where the last equality uses the definition of the $\theta$--rotation quandle
on $S^2$ (Example~\ref{theta}).
Therefore, $\iota$ is a quandle isomorphism.
\end{proof}

\begin{prop}
Let $S^2_\theta$ be the $\theta$--rotation quandle on the $2$--sphere.  
Then the following statements hold:
\begin{enumerate}
    \item If $0 < \theta < 2\pi$ and $\theta \neq \pi$, then the quandle $S^2_\theta$
    admits an embedding into the conjugation quandle $\Conj(SO(3))$.
    \begin{align*}
        \iota\::\:S^2_\theta&\rightarrow \Conj(SO(3))\\
        SO(2)g&\mapsto g^{-1}h_\theta g
    \end{align*}

    \item If $\theta = \pi$, then $S^2_\pi$ admits an embedding into the conjugation
    quandle $\Conj(Spin(3))$.
    \begin{equation*}
        \iota\::\:S^2_\pi\rightarrow \Conj(Spin(3))
    \end{equation*}
\end{enumerate}
\end{prop}
\begin{proof}
We treat the two cases separately.

\medskip
\noindent
\textbf{Case 1:} $0 < \theta < 2\pi$, $\theta \neq \pi$.

In this case, the fixed point subgroup of $\sigma_\theta$ is
\[
\Fix(\sigma_\theta) = SO(2) \subset SO(3).
\]
By Theorem~\ref{Main} and Lemma \ref{hodai.S^2_theta}, the homogeneous quandle $S^2_\theta=Q(SO(3),SO(2),\sigma_\theta)$ admits an embedding into
the conjugation quandle $\Conj(SO(3))$, 
\begin{align*}
        \iota\::\:S^2_\theta&\rightarrow \Conj(SO(3))\\
        SO(2)g&\mapsto g^{-1}h_\theta g.
    \end{align*}

\medskip
\noindent
\textbf{Case 2:} $\theta = \pi$.

In this case, the fixed point subgroup satisfies
\[
\Fix(\sigma_\pi) \supsetneqq SO(2),
\]
so it is no longer sufficient to embed into $SO(3)$.  
Instead, we consider the universal cover $Spin(3)$ of $SO(3)$.  
Moreover, for $\theta = \pi$, the quandle $S^2_\pi$ is isomorphic to the
spherical quandle $S^2$, as in Example~\ref{theta}.  
As will be shown in Remark~\ref{GrtoS^n} and Theorem~\ref{Grassmann embedding},
$S^2_\pi$ is therefore isomorphic to a homogeneous quandle on
$\widetilde{{Gr}}(3,1)$ and admits an embedding into $\Conj(Spin(3))$, 

\begin{align*}
    \iota\::\:S^2_\pi &\rightarrow \Conj(Spin(3))\\
    \mr{Stab}_{{Spin}(3)}(\widetilde{h}_{(3,1)})\,\widetilde{g}
    &\mapsto \widetilde{g}^{-1}\widetilde{h}_{(3,1)} \widetilde{g},
\end{align*}
% \begin{align*}
%         \iota\::\:S^2_\pi&\rightarrow \Conj(Spin(3))\\
%         \mr{Stab}(\widetilde{h}_{(3,1)})\widetilde{g}&\mapsto \widetilde{g}^{-1}\widetilde{h}_{(3,1)} \widetilde{g},
%     \end{align*}
    where $\widetilde{h}_{(3,1)}$ is a lift to $Spin(3)$ of
\[
h_{(3,1)} :=
\begin{pmatrix}
1 & 0 & 0\\
0 & -1 & 0\\
0 & 0 & -1
\end{pmatrix}
\in SO(3),
\]
$Spin(3)$ acts on itself by conjugation, and $\mr{Stab}_{{Spin}(3)}(\widetilde{h}_{(3,1)})$ is its stabilizer subgroup.

\medskip

Combining the two cases, the proposition follows.
\end{proof}

\begin{rem}
For any $\theta \in (0,2\pi)$ with $\theta \neq \pi$, 
the quandle $S^2_\theta$ can also be embedded into $\Conj(Spin(3))$.
Indeed, the universal covering map $p : Spin(3) \to SO(3)$ 
lifts the embedding $S^2_\theta \to \Conj(SO(3))$, 
providing a canonical embedding into $\Conj(Spin(3))$.
Since $Spin(3)$ is isomorphic to $SU(2)$, 
this embedding agrees with the realization of spherical quandles
as conjugation quandles in $SU(2)$ described in Lemma~4.4 of Clark and Saito~\cite{ClarkSaito}.
\end{rem}

\begin{rem}
One can similarly define a quandle structure on the oriented Grassmannian
$\widetilde{Gr}(n,n-2)$ by considering $\theta$-rotations.
By an analogous argument, it is expected that this quandle admits an embedding
into $\Conj(SO(n))$ or $\Conj(Spin(n))$.
\end{rem}

\subsection{Embedding of the unoriented Grassmann quandle}\label{unoriented Grassmann quandle}
We investigate an embedding of the Grassmann quandle introduced in Example $\ref{Grassmannquandle}$.  
For this purpose, we first discribe it by means of a quandle triplet and then describe the induced quandle structure. Let $n$ and $k$ be integers with $1\le k\le n$.

\begin{lemma}\label{lemma:GrassmannHomogeneous}
Consider the Grassmann quandle $Gr(n,k)$.  
Define an automorphism $\sigma=\mr{Ad}(h_{(n,k)}): O(n)\rightarrow O(n)$ for  
\[
h_{(n,k)} = \begin{pmatrix}
    E_k&0\\
    0&-E_{n-k}
\end{pmatrix}, \]namely
\[
\sigma(g) = h_{(n,k)}^{-1}\, g\, h_{(n,k)}\; \text{for} \; g \in O(n),
\]
Then $(O(n),\, O(k)\times O(n-k),\, \sigma\bigr)$ is a quandle triplet and there is a quandle isomorphism
\[
Gr(n,k) \cong Q\bigl(O(n),\, O(k)\times O(n-k),\, \sigma\bigr).
\]
\end{lemma}
\begin{proof}
Let $e_1,\dots,e_n$ be the standard basis of $\mathbb R^n$ and set
\[
V_0:=\mr{span}\{e_1,\dots,e_k\}\in Gr(n,k).
\]
The orthogonal group $O(n)$ acts transitively on $Gr(n,k)$ in natural way.
The stabilizer subgroup of $V_0$ is $O(k)\times O(n-k)$.
Hence we have a homogeneous space description
\[
Gr(n,k)\cong (O(k)\times O(n-k))\backslash O(n)
\]
by the map
\begin{align*}
    \iota: Gr(n,k)&\rightarrow (O(k)\times O(n-k))\backslash O(n)\\
    V=V_0g=\mr{span}\{e_1g,\dots,e_kg\}&\mapsto
\iota(V)=(O(k)\times O(n-k))\,g.
\end{align*}

Next we verify that $\iota$ is a quandle homomorphism.
Let $V=V_0 g$ and $W=V_0 h$.
The Grassmann quandle operation is given by the reflection of $V$
with respect to $W$, hence
\begin{align*}
V*W
&= \mr{span}\Bigl\{
2\sum_{j=1}^k \langle e_j h , e_i g\rangle e_j h - e_i g
\ \big|\ i=1,\dots,k
\Bigr\}\\
 &=\mr{span}(e_1 g h^{-1} h_{(n,k)} h,\dots,
                e_k g h^{-1} h_{(n,k)} h)\\
&=V_0gh^{-1}h_{(n,k)}h.
\end{align*}
This computation shows that
\begin{align*}
    \iota(V*W)
    &=\iota\bigl(
      V_0gh^{-1}h_{(n,k)}h
      \bigr)\\
    &=(O(k)\times O(n-k))\, g h^{-1} h_{(n,k)} h.
\end{align*}

On the other hand, since
\[
\sigma(gh^{-1})h
= h_{(n,k)}^{-1}\, g h^{-1}\, h_{(n,k)}\, h,
\]
and $h_{(n,k)}\in O(k)\times O(n-k)$, it follows that
\[
(O(k)\times O(n-k))\,\sigma(gh^{-1})h
=
(O(k)\times O(n-k))\, g h^{-1} h_{(n,k)} h.
\]
Hence $\iota$ is a quandle isomorphism.
\end{proof}

% \begin{prop}[Embedding of the Grassmann quandle]
% The Grassmann quandle $Gr(n,k)$ admits an embedding into the conjugation quandle $\Conj(O(n))$.
% \begin{equation*}
%     Gr(n, k)\rightarrow \Conj (O(n))
% \end{equation*}
% \end{prop}
\begin{prop}[Embedding of the Grassmann quandle]
The map
\begin{align*}
     \iota \colon Gr(n,k) &\rightarrow \Conj(O(n))\\
     V_0g&\mapsto g^{-1}h_{(n, k)}g
\end{align*}
is an embedding of quandle.
\end{prop}

\begin{proof}
By Lemma~\ref{lemma:GrassmannHomogeneous}, we have
\[
Gr(n,k)\cong Q\bigl(O(n),O(k)\times O(n-k),\sigma\bigr),
\quad
\sigma=\mr{Ad}(h_{(n,k)}).
\]
Since
\[
\Fix(\sigma)=O(k)\times O(n-k),
\]
and $\sigma\in\Inn(O(n))$,
Theorem~\ref{Main} implies that
$Q(O(n),O(k)\times O(n-k),\sigma)$  can be embedded into $\Conj(O(n))$.
Under the above identification, the Grassmann quandle $Gr(n, k)$ can be embedded into $\Conj(O(n))$ by the map
% \[
% f\colon Gr(n,k)\to \Conj(O(n))
% \]
% is explicitly given by
\[
\iota(V_0g)=g^{-1}h_{(n,k)}g,
\qquad g\in O(n).
\]
\end{proof}

\begin{rem}
When $k=1$, the Grassmann manifold $Gr(n,1)$ is naturally identified
with the real projective space $\mathbb{R}P^{n-1}$.
In this case, the above embedding
\[
Gr(n,1)\cong Q\bigl(O(n),\, O(1)\times O(n-1),\, \sigma\bigr)
\rightarrow \Conj (O(n))
\]
coincides with the embedding of $\mathbb{R}P^{n-1}$ into
$O(n)$ proposed by Eisermann\cite{Eisermann} and constructed by Yonemura\cite{Yonemura}.
\end{rem}

% \textcolor{red}{oriented quandle, element kigou dousiyo konosita section gaiikana?}

\subsection{Embedding of the oriented Grassmann quandle}\label{oriented grassmann}
In this section, we study embeddings of the oriented Grassmann quandle.
Although the oriented Grassmann manifold
\(\widetilde{Gr}(n,k)\) is always realized as the homogeneous space
\((SO(k)\times SO(n-k))\backslash SO(n)\),
the construction of the quandle structure and the target of the embedding
depend on the parity of \(n-k\).

For the oriented Grassmann manifold $\widetilde{Gr}(n,k)$,
we write
\[
\mathrm{span}\langle v_1,\ldots,v_k\rangle
\]
for the oriented $k$--dimensional subspace determined by an
ordered orthonormal frame $(v_1,\ldots,v_k)$.
The angle bracket $\langle\cdots\rangle$ is used to distinguish oriented subspaces
from the underlying unoriented subspace $\mr{span}(v_1,\cdots, v_k)$.

%We first treat the case where \(n-k\) is even, and then discuss the case
%where \(n-k\) is odd.
%\textcolor{red}{here syuuron hukkatu this map de isom. only}
% The oriented Grassmann manifold \(\widetilde{Gr}(n,k)\)
% admits the homogeneous space description
% \[
% \widetilde{Gr}(n,k)
% \cong SO(n)/(SO(k)\times SO(n-k)).
% \]

\begin{lemma}\label{oriGr,isom}
Assume that \(n-k\) is even.
Let
\[
h_{(n,k)}=
\begin{pmatrix}
E_k & 0\\
0 & -E_{n-k}
\end{pmatrix}
\in SO(n),
\]
and define an automorphism
\(\sigma=\mr{Ad}(h_{(n,k)}): SO(n)\rightarrow SO(n)\).
Then there exists a quandle isomorphism
\[
\widetilde{Gr}(n,k)
\cong Q\bigl(SO(n),\,SO(k)\times SO(n-k),\,\sigma\bigr).
\]
\end{lemma}
%     The special orthogonal group \(SO(n)\) acts on the oriented Grassmann manifold
% \(\widetilde{Gr}(n,k)\) transitively in natural way. For an oriented \(k\)-frame \((v_1,\cdots,v_k)\) representing
% \(V=\operatorname{span}(v_1,\cdots,v_k)\in\widetilde{Gr}(n,k)\) and for
% \(g\in SO(n)\), we define
% \[
% V\cdot g
% = \operatorname{span}\langle v_1 g,\cdots,v_k g\rangle.
% \]
% This defines a transitive right action of \(SO(n)\) on
% \(\widetilde{Gr}(n,k)\).
% Let \(V_0=\operatorname{span}\langle e_1,\cdots,e_k\rangle\).
% Then its stabilizer is given by
% \[
% \mathrm{Stab}_{SO(n)}(V_0)=SO(k)\times SO(n-k).
% \]
% It gives an identification between $\widetilde{Gr}(n,k)$ and \((SO(k)\times SO(n-k))\backslash SO(n)\) given by the map
% \begin{align*}
%     \iota:\widetilde{Gr}(n,k)&\rightarrow(SO(k)\times SO(n-k))\backslash SO(n)\\
%     V=V_0g=\operatorname{span}\langle e_1g,\cdots, e_kg\rangle&\mapsto \iota(V)=(SO(k)\times SO(n-k))g,
% \end{align*}
% where $V_0=\mr{span}\langle e_1, \cdots, e_k\rangle$ and $\mr{Stab}(V_0)=SO(k)\times SO(n-k)$.
% For an oriented \(k\)-frame \((v_1,\cdots,v_k)\) representing
% \(V=\operatorname{span}(v_1,\cdots,v_k)\in\widetilde{Gr}(n,k)\) and for
% \(g\in SO(n)\), we define
% \[
% V\cdot g=\operatorname{span}\langle v_1,\cdots,v_k\rangle\cdot g
% = \operatorname{span}\langle v_1 g,\cdots,v_k g\rangle.
% \]

%\begin{proof}
% The special orthogonal group \(SO(n)\) acts on the oriented Grassmann manifold
% \(\widetilde{Gr}(n,k)\) transitively in a natural way. 

\begin{proof}
For an oriented \(k\)-frame \((v_1,\cdots,v_k)\) representing
\(V=\operatorname{span}\langle v_1,\cdots,v_k\rangle\in\widetilde{Gr}(n,k)\) and for
\(g\in SO(n)\), we define
\[
V\cdot g
= \operatorname{span}\langle v_1 g,\cdots,v_k g\rangle.
\]
This defines a transitive right action of \(SO(n)\) on
\(\widetilde{Gr}(n,k)\).
Let \(V_0=\operatorname{span}\langle e_1,\cdots,e_k\rangle\).
Then its stabilizer of $SO(n)$-action at $V_0$ is given by
\[
\mathrm{Stab}_{SO(n)}(V_0)=SO(k)\times SO(n-k).
\]
Hence we obtain an identification
\[
\widetilde{Gr}(n,k)
\cong (SO(k)\times SO(n-k))\backslash SO(n)
\]
given by the map
\begin{align*}
\iota:\widetilde{Gr}(n,k)&\rightarrow(SO(k)\times SO(n-k))\backslash SO(n)\\
V=V_0g=\operatorname{span}\langle e_1g,\cdots, e_kg\rangle
&\mapsto (SO(k)\times SO(n-k))g.
\end{align*}
% This action is transitive, and the stabilizer of the standard oriented
% \(k\)-plane \(\operatorname{span}(e_1,\cdots,e_k)\) is given by
% \(SO(k)\times SO(n-k)\).
% Consequently, we obtain the identification
% \[
% \widetilde{Gr}(n,k)
% \cong (SO(k)\times SO(n-k))\backslash SO(n).
% \]
% Let
% \[
% V_0:=\operatorname{span}\{e_1,\dots,e_k\}\in \widetilde{Gr}(n,k).
% \]
% Define a map
% \[
% \iota :
% (SO(k)\times SO(n-k))\backslash SO(n)
% \longrightarrow \widetilde{Gr}(n,k)
% \]
% by
% \[
% \iota\bigl((SO(k)\times SO(n-k))\,g\bigr)
% := V_0 g
% = \operatorname{span}(e_1 g,\dots,e_k g).
% \]
% This map is well-defined and bijective.

We show that $\iota$ is a quandle homomorphism.
Let
\[
S:=SO(k)\times SO(n-k).
\]
Take arbitrary elements $Sg,Sh\in S\backslash SO(n)$.
By definition of $\iota$, we have
\[
\iota(Sg)=V_0 g, \qquad \iota(Sh)=V_0 h.
\]
The quandle operation on the homogeneous quandle
$Q(SO(n),S,\sigma)$ is given by
\[
(Sg)*(Sh)=S\,\sigma(g h^{-1})h.
\]
Therefore, since $h_{(n,k)}\in S$, hence $V_0h_{(n,k)}^{-1}=V_0$, we have
\begin{align*}
\iota\bigl((Sg)*(Sh)\bigr)
&=\iota\bigl(S\,\sigma(g h^{-1})h\bigr)\\
&=V_0\,\sigma(g h^{-1})h\\
&=V_0\, h_{(n,k)}^{-1} g h^{-1} h_{(n,k)} h.
\end{align*}
On the other hand, as we can see in the proof of Lemma  
\ref{lemma:GrassmannHomogeneous}, the Grassmann quandle operation on
$\widetilde{Gr}(n,k)$ is given by
\[
(V_0 g)*(V_0 h)
=V_0 g h^{-1} h_{(n,k)} h.
\]
Therefore,
\[
\iota\bigl((Sg)*(Sh)\bigr)
=\iota(Sg)*\iota(Sh),
\]
which shows that $\iota$ is a quandle homomorphism.
Since $\iota$ is bijective, it is a quandle isomorphism.
\end{proof}

By the above lemma, we have an isomorphism
\[
\widetilde{Gr}(n,k)
\cong Q\bigl(SO(n),\, SO(k)\times SO(n-k),\, \sigma\bigr).
\]
However, the fixed-point subgroup \(\mathrm{Fix}(\sigma)\) does not coincide with
\(SO(k)\times SO(n-k)\).
Therefore, it is necessary to pass to the
larger covering group \(Spin(n)\) or \(Pin^+(n)\).

We first treat the case where \(n-k\) is even.
Let
\[
p\colon Spin(n)\to SO(n)
\]
be the double covering map, and fix a lift
\[
\widetilde h_{(n,k)}\in Spin(n)
\]
of
\[
h_{(n,k)}
=
\begin{pmatrix}
E_k & 0\\
0 & -E_{n-k}
\end{pmatrix}
\in SO(n).
\]
Define
\[
\widetilde{\sigma}
=
\mathrm{Ad}(\widetilde h_{(n,k)}).
\]

The group $Spin(n)$ acts on the oriented Grassmannian
$\widetilde{Gr}(n,k)$ via the double covering
$p\colon Spin(n)\to SO(n)$.
Explicitly, for
\[
V=\operatorname{span}\langle v_1,\dots,v_k\rangle,
\]
we define
\[
V\cdot \widetilde g
=
\operatorname{span}
\langle v_1p(\widetilde g),\dots,v_kp(\widetilde g)\rangle.
\]

Let
\[
V_0=\operatorname{span}\langle e_1,\dots,e_k\rangle.
\]
We denote by
\[
\mathrm{Stab}_{Spin(n)}(V_0)
=
\{
\widetilde g\in Spin(n)
\mid
V_0\cdot \widetilde g=V_0
\}
\]
the stabilizer of $V_0$ under this action.
Then
\[
\mathrm{Stab}_{Spin(n)}(V_0)
=
p^{-1}(SO(k)\times SO(n-k)).
\]

On the other hand, let $C_{Spin(n)}(\tilde h_{(n,k)})$ be the centralizer of $\tilde h_{(n,k)}$ in $Spin(n)$. 
It is easy to see that 
\[
{\rm Fix}(\tilde\sigma)=C_{Spin(n)}(\tilde h_{(n,k)})\subset p^{-1}(SO(k)\times SO(n-k)). 
\]
Since the natural  multiplication map 
\[
Spin(k)\times Spin(n-k)\rightarrow Spin(n)
\]gives a surjective homomorphism onto $p^{-1}(SO(k)\times SO(n-k))$ and 
\[
\tilde g\tilde h_{(n,k)}\tilde g^{-1}\tilde h_{(n,k)}^{-1} \in {\rm ker}(p)\cong \{\pm 1\}
\]for any $\tilde g\in p^{-1}(SO(k)\times SO(n-k))$, the connectedness of $Spin(k)\times Spin(n-k)$ implies that 
$\tilde g\tilde h_{(n,k)}\tilde g^{-1}\tilde h_{(n,k)}^{-1}=1$. 
Therefore we have 
$C_{Spin(n)}(\tilde h_{(n,k)}) = p^{-1}(SO(k)\times SO(n-k))$.

\begin{prop}\label{Oriented grassmannspin}
The oriented Grassmann quandle $\widetilde{Gr}(n,k)$
is isomorphic, as a quandle, to the homogeneous quandle
\[
Q\bigl(
Spin(n),
C_{Spin(n)}(\widetilde h_{(n,k)}),
\widetilde{\sigma}
\bigr).
\]
\end{prop}

\begin{proof}
By the above discussion,
\[
\mathrm{Fix}(\widetilde{\sigma})
=
C_{Spin(n)}(\widetilde h_{(n,k)}).
\]
Hence
\[
\bigl(
Spin(n),
C_{Spin(n)}(\widetilde h_{(n,k)}),
\widetilde{\sigma}
\bigr)
\]
is a quandle triplet, and therefore
\[
Q\bigl(
Spin(n),
C_{Spin(n)}(\widetilde h_{(n,k)}),
\widetilde{\sigma}
\bigr)
\]
is a homogeneous quandle.

Since $Spin(n)$ is simply connected and
$C_{Spin(n)}(\widetilde h_{(n,k)})$ is closed and connected
by \cite[Theorem~3.4]{Borel},
it follows from \cite[Proposition~3.6]{Helgason} that the homogeneous space
\[
C_{Spin(n)}(\widetilde h_{(n,k)})\backslash Spin(n)
\]
is the universal covering manifold of
\[
(SO(k)\times SO(n-k))\backslash SO(n).
\]

Hence there exists a $Spin(n)$-equivariant diffeomorphism
\[
\Phi\colon
C_{Spin(n)}(\widetilde h_{(n,k)})\backslash Spin(n)
\;\xrightarrow{\;\cong\;}
(SO(k)\times SO(n-k))\backslash SO(n),
\]
given by
\[
\Phi\bigl(
C_{Spin(n)}(\widetilde h_{(n,k)})\,\widetilde g
\bigr)
=
(SO(k)\times SO(n-k))\,p(\widetilde g).
\]

Since $\Phi$ is $Spin(n)$-equivariant and intertwines
$\widetilde{\sigma}$ and $\sigma$,
it induces an isomorphism of the associated homogeneous quandles.
Therefore, by Lemma~\ref{oriGr,isom},
the homogeneous quandle
\[
Q\bigl(
Spin(n),
C_{Spin(n)}(\widetilde h_{(n,k)}),
\widetilde{\sigma}
\bigr)
\]
is naturally identified with the oriented Grassmann quandle
$\widetilde{Gr}(n,k)$.
\end{proof}

By definition of $\widetilde{\sigma}$, we have $\mathrm{Fix}(\widetilde\sigma)=C_{Spin(n)}(\widetilde{h}_{(n,k)})$. 
In particular we have the following. 
\begin{theorem}\label{Grassmann embedding}
Assume that $n-k$ is even.
The oriented Grassmann quandle $\widetilde{Gr}(n,k)$
admits a quandle embedding
\begin{align*}
    \widetilde{Gr}(n,k)
&\rightarrow
\Conj(Spin(n))\\
V_0\widetilde{g}&\mapsto \widetilde{g}^{-1}\widetilde{h}_{(n,k)}\widetilde{g}
\end{align*}
\end{theorem}
Next, we consider the case where $n-k$ is odd.
In contrast to the even case, the involution $\sigma$ of $SO(n)$
does not lift to an inner automorphism of $Spin(n)$.
% Indeed, in this case we have $h_{(n,k)}\notin SO(n)$.
% Therefore the involution $\sigma(g)=h_{(n,k)}^{-1} g h_{(n,k)}$
% cannot be realized as an inner automorphism of $Spin(n)$.
However, $\sigma$ also gives an inner automorphism of $O(n)$ and it lifts to an inner automorphism 
\begin{equation*}
    \widetilde{\sigma}:=\mr{Ad}(\widetilde h_{(n,k)}): Pin(n)\to Pin(n).
\end{equation*}
We define the centralizer of $\widetilde h_{(n,k)}$ in $Pin(n)$ by
\[
C_{Pin(n)}(\widetilde h_{(n,k)})
=
\{\, \widetilde g\in Pin(n)\mid 
\widetilde g\widetilde h_{(n,k)}=\widetilde h_{(n,k)}\widetilde g \,\}.
\]
Similarly, we denote
\[
C_{Spin(n)}(\widetilde h_{(n,k)})
=
C_{Pin(n)}(\widetilde h_{(n,k)})\cap Spin(n)=\mr{Stab}_{Spin(n)}(V_0).
\]
The inclusion \(Spin(n)\subset Pin(n)\) induces a map
\begin{align*}
    \Psi\colon
C_{Spin(n)}(\widetilde h_{(n,k)})\backslash Spin(n)
&\rightarrow
C_{Pin(n)}(\widetilde h_{(n,k)})\backslash Pin(n)\\
C_{Spin(n)}(\widetilde{h}_{(n,k)})\,\widetilde g
&\mapsto
C_{Pin(n)}(\widetilde{h}_{(n,k)})\,\widetilde g .
\end{align*}

\begin{lemma}\label{lemma:PsiBijection}
The map $\Psi$ is a bijection.
\end{lemma}

\begin{proof}
The map $\Psi$ is well-defined since
$C_{Spin(n)}(\widetilde h_{(n,k)})
\subset
C_{Pin(n)}(\widetilde h_{(n,k)})$.

To prove injectivity, suppose that
\[
C_{Pin(n)}(\widetilde{h}_{(n,k)})\,\widetilde g
=
C_{Pin(n)}(\widetilde{h}_{(n,k)})\,\widetilde g'
\]
for $\widetilde g,\widetilde g'\in Spin(n)$.
Then $\widetilde g' \widetilde g^{-1}\in C_{Pin(n)}(\widetilde h_{(n,k)})\cap Spin(n)=C_{Spin(n)}(\widetilde h_{(n,k)})$.
Therefore,
\(
C_{Spin(n)}(\widetilde h_{(n,k)})\,\widetilde g
=
C_{Spin(n)}(\widetilde h_{(n,k)})\,\widetilde g'
\),
so $\Psi$ is injective.

% Surjectivity follows from the fact that the natural actions of
% $Spin(n)$ and $Pin(n)$ on the oriented Grassmann manifold
% $\widetilde{Gr}(n,k)$ are transitive.

To show surjectivity, let 
\[
C_{Pin(n)}(\widetilde h_{(n,k)}) \, \widetilde g' \in C_{Pin(n)}(\widetilde h_{(n,k)}) \backslash Pin(n)
\]
be an arbitrary coset.  
Since the action of ${Spin}(n)$ on the oriented Grassmann manifold 
$\widetilde{Gr}(n,k)$ is transitive, there exists 
$\widetilde g \in {Spin}(n)$ such that
\[
\widetilde g \cdot V_0 = \widetilde g' \cdot V_0,
\]
where $V_0$ is the reference oriented subspace.  
By definition of the stabilizer, there exists 
$h \in \mathrm{Stab}_{Pin}(\widetilde h_{(n,k)})$ such that
\[
\widetilde g' = h \widetilde g.
\]
Hence
\[
\Psi(C_{Spin(n)}(\widetilde h_{(n,k)}) \, \widetilde g) 
= C_{Pin(n)}(\widetilde h_{(n,k)}) \, \widetilde g 
= C_{Pin(n)}(\widetilde h_{(n,k)}) \, \widetilde g' ,
\]
showing that $\Psi$ is surjective.  

Consequently, every coset in 
\(C_{Pin(n)}(\widetilde h_{(n,k)})\backslash Pin(n)\)
has a representative in \({Spin}(n)\).
% Since the action of $\mathrm{Spin}(n)$ on the oriented Grassmann manifold $\widetilde{Gr}(n,k)$ is transitive, 
% there exists $\widetilde g \in \mathrm{Spin}(n)$ such that
% \[
% \widetilde g \cdot V_0 = \widetilde g' \cdot V_0,
% \]
% where $V_0$ is the reference oriented subspace. 

% By definition of the map $\Psi$, we have
% \[
% \Psi(\mathrm{Stab}_{Spin}(\widetilde h_{(n,k)})\, \widetilde g) 
% = \mathrm{Stab}_{Pin}(\widetilde h_{(n,k)})\, \widetilde g 
% = \mathrm{Stab}_{Pin}(\widetilde h_{(n,k)})\, \widetilde g' .
% \]
% Since the coset was arbitrary, $\Psi$ is surjective.
% Hence every coset in
% $\mathrm{Stab}_{Pin}(\widetilde h_{(n,k)})\backslash Pin(n)$
% has a representative in $Spin(n)$.

\end{proof}

% Since the natural actions of $Spin(n)$ and $Pin(n)$ on the oriented
% Grassmann manifold $\widetilde{Gr}(n,k)$ have the same orbits,
% both homogeneous spaces
% \[
% \mathrm{Stab}_{Spin}(\widetilde h_{(n,k)})\backslash Spin(n),
% \qquad
% \mathrm{Stab}_{Pin}(\widetilde h_{(n,k)})\backslash Pin(n)
% \]
% are naturally identified with $\widetilde{Gr}(n,k)$.
% By definition, the quandle structure on
% $Q\bigl(Spin(n),\mathrm{Stab}_{Spin}(\widetilde h_{(n,k)}),\widetilde\sigma\bigr)$
% is induced by the conjugation automorphism $\widetilde\sigma$,
% and the identification given by $\Psi$ is $Spin(n)$--equivariant.
% Therefore, the bijection $\Psi$ yields an isomorphism of homogeneous quandles,
% and hence
% \[
% Q\bigl(Pin(n),\mathrm{Stab}_{Pin}(\widetilde h_{(n,k)}),\widetilde\sigma\bigr)
% \cong
% \widetilde{Gr}(n,k).
% \]
% as quandles.
Since the homogeneous quandles
$Q\bigl(Spin(n),C_{Spin(n)}(\widetilde h_{(n,k)}),\sigma|_{Spin(n)}\bigr)$ and \\
$Q\bigl(Pin(n),C_{Pin(n)}(\widetilde h_{(n,k)}),\sigma\bigr)$
are defined by the conjugation automorphism $\sigma$ and its restriction
to $Spin(n)$, and since the above bijection
\[
C_{Spin(n)}(\widetilde h_{(n,k)})\backslash Spin(n)
\;\cong\;
C_{Pin(n)}(\widetilde h_{(n,k)})\backslash Pin(n)
\]
is an isomorphism of homogeneous spaces induced by the natural actions
of $Spin(n)$ and $Pin(n)$ on the oriented Grassmann manifold
$\widetilde{Gr}(n,k)$,
it follows that the two homogeneous quandles are isomorphic.
Since $\widetilde{\sigma}$ is an inner automorphism of $Pin(n)$
and $\Fix(\widetilde{\sigma})=C_{Pin(n)}(\widetilde h_{(n,k)})$,
Theorem \ref{Main} implies the following. 
\begin{theorem}
    Assume that $n-k$ is odd. The oriented Grassmann quandle admits an embedding
\begin{align*}
    \widetilde{Gr}(n,k)
&\rightarrow
\Conj(Pin(n))\\
V_0\widetilde{g}&\mapsto \widetilde{g}^{-1}\widetilde{h}_{(n,k)}\widetilde{g}.
\end{align*}

\end{theorem}

\begin{rem}\label{GrtoS^n}
    When $k=1$, the oriented Grassmann manifold
$\widetilde{{Gr}}(n,1)$ is naturally identified with the sphere $S^{n-1}$,
and the corresponding quandle structure coincides with the spherical quandle.
Moreover, the embedding constructed above reduces to the embedding
described in the theorem of Eisermann\cite{Eisermann} and Yonemura\cite{Yonemura}.
\end{rem}

\section*{Acknowledgements}
This paper is a part of the author's master's thesis.
The author would like to express her sincere gratitude to her supervisor, 
Hajime Fujita, for his continuous guidance and encouragement
throughout this research.
The author is also grateful to Hiroshi Tamaru
for introducing the notion of homogeneous quandles and for his important
suggestions on the direction of this research.
The author thanks Ryoya Kai
for helpful discussions on the case where the automorphism $\sigma$ in the main
theorem is not an inner automorphism, as well as for his insights into future
problems.
Finally, the author is grateful to Kentaro Yonemura
for kindly answering questions and for introducing many important references.

% -----------------------------
% 参考文献
% -----------------------------

\end{document}